\documentclass[psamsfonts,a4paper]{amsart}

\usepackage{amssymb,amsfonts}
\usepackage{amsaddr}
\usepackage[all,arc]{xy}
\usepackage{enumerate}
\usepackage{tikz}
\usepackage{parskip}
\usepackage[margin=2.5cm]{geometry}
\usepackage{eucal}

\newtheorem{thm}{Theorem}[section]

\newtheorem{cor}[thm]{Corollary}
\newtheorem{prop}[thm]{Proposition}
\newtheorem{lemma}[thm]{Lemma}

\theoremstyle{definition}

\newtheorem{nolabel}[thm]{}

\theoremstyle{remark}
\newtheorem{rem}[thm]{Remark}

\makeatletter
\let\c@equation\c@thm
\makeatother
\numberwithin{equation}{section}

\DeclareMathOperator{\en}{End}

\DeclareMathOperator{\Imag}{Im}

\DeclareMathOperator{\res}{Res}

\DeclareMathOperator{\zhu}{Zhu}

\newcommand{\C}{\mathbb{C}}

\newcommand{\Z}{\mathbb{Z}}

\newcommand{\HH}{\mathcal{H}}

\newcommand{\al}{\alpha}

\newcommand{\la}{\lambda}
\newcommand{\La}{\Lambda}

\newcommand{\ov}{\overline}

\newcommand{\vac}{\mathbf{1}}

\title[A short construction of the Zhu Algebra]{A short construction of the Zhu Algebra}
\author[*]{Jethro van Ekeren}
\author[**]{Reimundo Heluani}

\begin{document}

%% ----------------------------------------------------- %%

\begin{center}
{\LARGE \bf A short construction of the Zhu Algebra} \par \bigskip

\renewcommand*{\thefootnote}{\fnsymbol{footnote}}
{\normalsize
Jethro van Ekeren\footnote{email: \texttt{jethrovanekeren@gmail.com}}\textsuperscript{1},
Reimundo Heluani\textsuperscript{2}
}

\par \bigskip

\textsuperscript{1}{\footnotesize Instituto de Matem\'{a}tica e Estat\'{i}stica (GMA), UFF, Niter\'{o}i RJ, Brazil}

\par
\textsuperscript{2}{\footnotesize Instituto Nacional de Matem\'{a}tica Pura e Aplicada, Rio de Janeiro, RJ, Brazil}

\par \bigskip
\end{center}

\vspace*{10mm}

\noindent
\textbf{Abstract.} We investigate associative quotients of vertex algebras. We also give a short construction of the Zhu algebra, and a proof of its associativity using elliptic functions.

\vspace*{10mm}

\section{Introduction}\label{section.intro}

\begin{nolabel}
Let $V$ be a vertex algebra. In \cite{Zhu-thesis} and \cite{Z96} Y. Zhu introduced two associative algebras $\zhu(V)$ and $R_V$ constructed as quotients of $V$ equipped with products defined explicitly in terms of its $n^{\text{th}}$ products. The purpose of this note is to present a brief and uniform construction of $\zhu(V)$ and $R_V$ as well as a short proof of the associativity of $\zhu(V)$ utilising elliptic functions.

The representation theory of $\zhu(V)$ is closely connected with that of $V$ (for $V$ a conformal vertex algebra now). There are adjoint functors between the categories of $\zhu(V)$-modules and of positive energy $V$-modules, and these functors induce bijections between the sets of isomorphism classes of simple objects in each category. The existence of an associative algebra with these properties is not in itself surprising since the category of all $V$-modules is equivalent to the category of smooth modules of its enveloping algebra $U(V) = \bigoplus_{n \in \Z} U(V)_n$, and so the algebra $\widetilde{U} = UU(V)_0 / (U(V) \cdot U(V)_{>0})_0$ does the job. And indeed $\zhu(V) \cong \widetilde{U}$. What is nontrivial is the presentation of $\zhu(V)$ as an explicit quotient of $V$, which lends itself to practical applications.

The algebra $\zhu(V)$ was introduced by Zhu in the context of his study of the behaviour of conformal blocks for singular families of elliptic curves with nodal singular fibre. Our use of elliptic functions to prove associativity of $\zhu(V)$ lies closer, in fact, to this point of view on the latter than to its representation theoretic interpretation.
 
\end{nolabel} 

\begin{nolabel}
Let $V$ be a vertex algebra and $f \in \mathbb{C}\left( (t) \right)$ a Laurent series. We have a bilinear operation $V \otimes V \rightarrow V$, \[ a \otimes b \mapsto a_{(f)}b := \res_t f(t)a(t)b \]
which we refer to as the $(f)$-product. We study quotients of the form $Z_f(V) := V/\left(V_{(f)}(T V ) \right)$ equipped with the algebraic structure of $(f)$-product. Here $V_{(f)}(T V )$ is the  ideal in $V$ generated by $TV$ under the $(f)$-product. It turns out that associativity of such quotients (when well defined) imposes very restrictive conditions on $f$. Let $f(t) = \zeta(t, \omega_1, \omega_2)$ be Weierstrass' zeta function, depending on the two periods $\omega_{1},\omega_2 \in \mathbb{C}$. This function does not quite satisfy the conditions guaranteeing associativity of $Z_f(V)$, but its limit $\omega_1 \rightarrow \infty$ does satisfy these conditions, and one recovers Zhu's algebra $\zhu(V)$ (or rather an isomorphic algebra found by Huang \cite{Huang.duality} ). On the other hand the limit $\omega_2 \mapsto \infty$ recovers Zhu's $C_2$-algebra $R_V$.
\label{no:limits}
\end{nolabel}

%The role played by $R_V$ in Zhu's work is more technical. Finite dimensionality of $R_V$ guarantees the existence of differential equations of Fuchsian type satisfied by the torus $1$-point functions of $V$ and consequently the convergence of the formal series defining the latter. Subsequent work has nevertheless revealed that $R_V$ and the scheme $X_V = \spec{R_V}$ known as the \emph{associated variety} have representation theoretic significance. Miyamoto showed that $\dim{R_V} < \infty$ is enough to guarantee modular behaviour of torus $1$-point functions \cite{MiyaC2}, and work of Arakawa has made $X_V$ into a useful tool in the representation theory of $W$-algebras \cite{Arakawa.assoc}. Mention Higgs Branch maybe ????

%% The grading of $V$ by eigenvalues of $L_0$ induces a grading of $R_V$, and a filtration of $\zhu(V)$. There is a surjection
%% \[
%% R_V \rightarrow \gr^\bullet \zhu(V).
%% \]

\section{Vertex Algebras}\label{section.notation}

For the basics of formal calculus and the definition of vertex algebra we refer to Kac's book \cite{Kac.VA.Book}. The formal delta function is defined to be
\[
\delta(z, w) = \sum_{n \in \Z} z^{-n-1} w^n \in \C[[z^{\pm 1}, w^{\pm 1}]].
\]
The ring morphism
\[
i_{z,w} : \C[[z, w]][z^{-1}, w^{-1}, (z-w)^{-1}] \rightarrow \C((z))((w))
\]
is defined by Taylor expansion of $(z-w)^n$ in positive powers of $w$. Note that the image of $i_{w,z}$ is $\C((w))((z))$ which is distinct from $\C((z))((w))$. We denote differentiation by $\partial$, appending the variable as a subscript when necessary to avoid confusion. We abbreviate $\tfrac{1}{j!}\partial^j$ as $\partial^{(j)}$. The formal residue of $f(z) = \sum_{n\in \Z} f_n z^n$ is defined to be
\[
\res_z f(z) dz = f_{-1}.
\]
The relations 
\begin{align}\label{totalderivative}
\partial_w^{(j)}\delta(z,w) = (-1)^j \partial_z^{(j)}\delta(z,w) \quad \text{and} \quad \res_z\left( f(z)\partial_z^{(j)}\delta(z, w) \right) dz = (-1)^j \partial_w^{(j)}f(w).
\end{align}
are easily verified.

A vertex algebra is a vector space $V$ together with a linear map
\[
V \otimes V \rightarrow V((z)), \quad \text{written} \quad a \otimes b \mapsto a(z)b = \sum_{n\in \Z} a(n)b z^{-n-1},
\]
as well as a vacuum vector $\vac$ and translation operator $T \in \en(V)$, which satisfy the Borcherds identity
\begin{align}\label{borfor}
[a(x)b](w) = \res_z\left(a(z)b(w)i_{z,w} - b(w)a(z)i_{w,z}\right)\delta(x,z-w)dz,
\end{align}
the vacuum axiom
\[
\vac(z) = 1_V \quad \text{and} \quad a(z)\vac \in a + zV[[z]],
\]
and the translation invariance axiom
\[
[Ta](z) = \partial_z a(z).
\]
Some useful consequences of the axioms are the commutator formula
\begin{align}\label{commfor}
[a(z),b(w)] = \sum_{j \in \Z_+} [a(j)b](w)\partial_w^{(j)}\delta(z,w),
\end{align}
and the skew-symmetry formula
\begin{align}\label{sksym}
a(z)b = e^{zT}b(-z)a.
\end{align}

\section{Associative Quotients of Vertex Algebras}\label{section.quotients}

Let $V$ be a vertex algebra, and let $f$ be a Laurent series. For any $a, b \in V$ we denote by $a_{(f)}b$ the vector $\res_z f(z)a(z)b \,dz \in V$. For example $a_{(1)}b = a(0)b$ and $(Ta)_{(f)}b = -a_{(\partial f)}b$. We call the bilinear operation
\[
V \otimes V \rightarrow V, \quad a \otimes b \mapsto a_{(f)}b
\]
the $f$-product. Sometimes we abuse this notation, writing for example $a_{(f(-z))}b$ to designate the product $a_{(g)}b$ where $g(z) = f(-z)$.

In this section we study $f$-products and subspaces of $V$ of the form
\[
V_{(f)}V = \left< a_{(f)}b | a, b \in V \right>,
\]
where here, and throughout the paper, $\left<\cdots \right>$ denotes $\C$-linear span. In particular, we should like to determine natural conditions on Laurent series $f, g$ under which $V_{(g)}V$ becomes an ideal for the $f$-product.

Direct computation with the commutator formula (\ref{commfor}) yields
\begin{align}\label{idealstart}
\begin{split}
a_{(f)}(b_{(g)}c)
&= \res_z \res_w f(z) g(w) a(z)b(w) c \,dw \,dz \\
&= \res_z \res_w f(z) g(w) b(w)a(z) c \,dw \,dz + \res_z \res_w f(z) g(w) [a(z),b(w)] c \,dw \,dz \\
&= b_{(g)}(a_{(f)}c) + \sum_{j \in \Z_+} \res_z \res_w f(z) g(w) [a{(j)}b](w) c \partial_w^{(j)}\delta(z,w) \,dw \,dz.
\end{split}
\end{align}
Using (\ref{totalderivative}) the $j^{\text{th}}$ summand here simplifies to
\begin{align*}
(-1)^j \res_z \res_w f(z) g(w) [a{(j)}b](w) c \partial_z^{(j)}\delta(z,w) \,dw \,dz
&= \res_z \res_w (\partial_z^{(j)}f(z)) g(w) [a{(j)}b](w) c \delta(z,w) \,dw \,dz \\
&= \res_w (\partial_w^{(j)}f(w)) g(w) [a{(j)}b](w) c \,dw \,dz.
\end{align*}
%% \begin{align*}
%% \res_z \res_w f(z) g(w) [a{(j)}b](w) c \partial_w^{(j)}\delta(z,w)
%% %
%% &= (-1)^j \res_z \res_w f(z) g(w) [a{(j)}b](w) c \partial_z^{(j)}\delta(z,w) \\
%% %
%% &= \res_z \res_w (\partial_z^{(j)}f(z)) g(w) [a{(j)}b](w) c \delta(z,w) \\
%% %
%% &= \res_w (\partial_w^{(j)}f(w)) g(w) [a{(j)}b](w) c.
%% \end{align*}

The only universal way (that is, by placing conditions on $f$ and $g$ but not on $V$) to guarantee $a_{(f)}(b_{(g)}c) \in V_{(g)}V$ is to require
\begin{align}\label{leftideal.condition}
g(z) \partial_z^{(j)}f(z) \in \left<\partial_z^{(k)} g(z) | k \in \Z_+ \right> \quad \text{for all $j \in \Z_+$}.
\end{align}

%% {\color{red}One might wish to require it only for $j \geq 1$, and somehow deal with
%% \[
%% (a{(0)}b)_{(fg)}c.
%% \]
%% }

Similarly, direct computation with the Borcherds identity (\ref{borfor}) yields
\begin{align}\label{assocstart}
  \begin{split}
(a_{(f)}b)_{(f)}c
&= \res_x \res_w f(x) f(w) [a(x)b](w) c \,dw \,dx \\
&= \res_x \res_w \res_z f(x) f(w) \left(a(z)b(w)c i_{z,w} - b(w)a(z)c i_{w,z}\right)\delta(x,z-w) \,dz \,dw \,dx \\
&= \res_w \res_z f(w) \left(a(z)b(w)c i_{z,w} - b(w)a(z)c i_{w,z}\right) f(z-w) \,dz \,dw.
\end{split}
\end{align}
Substituting the expansions
\[
i_{z,w}f(z-w) = \sum_{j \in \Z_+} (-w)^j f^{(j)}(z) \quad \text{and} \quad i_{w,z}f(z-w) = \sum_{j \in \Z_+} z^j f^{(j)}(-w)
\]
into the final line of (\ref{assocstart}) reduces the two summands there to
\begin{align*}
a_{(f)}(b_{(f)}c) + \sum_{j \in \Z_{\geq 1}} (-1)^j a_{(\partial^{(j)}f)}(b_{(z^j f(z))}c) \quad \text{and} \quad -\sum_{j \in \Z_+} b_{(\partial^{(j)}f(-z) f(z))}(a(j)c),
\end{align*}
respectively.
% {\color{red}Awkward notation here. Need functional notation for $z^n$ and $f(-z)$}

% \begin{align*}
% &= \res_w \res_z f(w) a(z)b(w)c i_{z,w} f(z-w) \\
% %
% &= \sum_{j \geq 0} \res_w \res_z  (-w)^j [\partial_z^{(j)} f](z) f(w) a(z)b(w)c \\
% %
% &= a_{(f)}(b_{(f)}c) + \sum_{j \geq 1} (-1)^j a_{(\partial^{(j)}f)}(b_{(w^j f)}c)
% \end{align*}

% \begin{align*}
% &= \res_w \res_z f(w) b(w)a(z)c i_{w,z}f(z-w) \\
% %
% &= \sum_{j \geq 0} \res_w \res_z z^j f^{(j)}(-w) f(w) b(w)a(z)c \\
% %
% &= \sum_{j \geq 0} \res_w f^{(j)}(-w) f(w) b(w) a(j)c \\
% %
% &= \sum_{j \geq 0} b_{(f^{(j)}(-w) f(w))}(a(j)c)
% \end{align*}

The relation
\[
(a_{(f)}b)_{(f)}c - a_{(f)}(b_{(f)}c) \in V_{(g)}V
\]
is thus assured by the conditions
\begin{align}\label{associativity.condition1}
\partial_z^{(j)}f(z) &\in \left<\partial_z^{(k)} g(z) | k \in \Z_+ \right> \quad \text{for all $j \in \Z_{\geq 1}$}, \\
\label{associativity.condition2} \text{and} \quad 
f(z) \partial_z^{(j)}f(-z) &\in \left<\partial_z^{(k)} g(z) | k \in \Z_+ \right> \quad \text{for all $j \in \Z_+$}.
\end{align}

%% \begin{align*}
%%   \begin{split}
%% (a_{(g)}b)_{(f)}c
%% %
%% &= \res_x \res_w g(x) f(w) [a(x)b](w) c \\
%% %
%% &= \res_x \res_w \res_z g(x) f(w) \left(a(z)b(w)c i_{z,w} - b(w)a(z)c i_{w,z}\right)\delta(x,z-w) \\
%% %
%% &= \res_w \res_z f(w) \left(a(z)b(w)c i_{z,w} - b(w)a(z)c i_{w,z}\right) g(z-w).
%% \end{split}
%% \end{align*}

By repeating the last calculation with $g(x)$ in place of $f(x)$ we see that the relation $(a_{(g)}b)_{(f)}c \in V_{(g)}V$ is assured by
\begin{align}\label{rightideal.condition}
f(z) \partial_z^{(j)}g(-z) &\in \left<\partial_z^{(k)} g(z) | k \in \Z_+ \right> \quad \text{for all $j \in \Z_+$}.
\end{align}

If the $z^{-1}$ coefficient of $g(z) = \sum g_n z^n$ is nonzero then $V_{(g)}V = V$. So we exclude this case as trivial. On the other hand we should like $\vac$ to be a unit of $V / V_{(g)}V$. But $\vac_{(f)}a = f_{-1} a$, so we are forced to require
\begin{align}\label{unit.condition}
f_{-1} = 1.
\end{align}
It follows from this and (\ref{leftideal.condition}) that $g(z)$ has nonzero singular part. 
%Suppose the lowest nontrivial coefficient of $f(z) = \sum f_n z^n$ is $f_N$

If we impose the condition
\begin{align}\label{non.higher.condition}
f(z) \in z^{-1} + \C[[z]]
\end{align}
then (\ref{associativity.condition1}) forces $g = \partial f$ (after a possible rescaling). Hence by (\ref{associativity.condition2}) we have
\begin{align}\label{f.ODE}
f(-z)f(z) = \partial_z f(z).
\end{align}
Observe that $f(-z)f(z)$ is an even function of $z$, hence $\partial_z f(z)$ is too, and so $f(z)$ takes the form $c/2 + F(z)$ for some constant $c$ and odd function $F(z)$. Then (\ref{f.ODE}) becomes
\[
\partial_z F(z) = c^2/4 - F(z)^2.
\]
Integrating this equation, and reimposing (\ref{unit.condition}), yields the two solutions
\begin{align}
f(z) &= z^{-1}, \label{C2.solution} \\
\text{or} \quad f(z) &= c \frac{e^{cz}}{e^{cz}-1} \quad \quad \text{for some $c \neq 0$}. \label{Zhu.solution}
\end{align}
%% \begin{align*}
%% f(z) &= Az^{-1} & \text{if $c = 0$}, \\
%% %
%% \text{and} \quad f(z) &= 2c \frac{Ae^{2cz}}{Ae^{2cz}-1} & \text{if $c \neq 0$}.
%% \end{align*}

% Now {\color{red}as noted by HUANG} this is isomorphic to the Zhu algebra. But the change of coordinates is due to Zhu. See if it is explained in \cite{Zhu96}

\begin{prop}\label{zhu.is.assoc}
For $f$ either solution \textup{(}\ref{C2.solution}\textup{)} or \textup{(}\ref{Zhu.solution}\textup{)} and $g = \partial f$, the quotient $V / V_{(g)}V$ is an associative algebra with respect to the $f$-product, and $\vac$ is a unit.
\end{prop}

\begin{proof}
Conditions (\ref{unit.condition}) and (\ref{associativity.condition1}) are satisfied already, as is (\ref{associativity.condition2}) for $j=0$. Condition (\ref{rightideal.condition}) follows from (\ref{associativity.condition2}) since $g = \partial f$. It remains to verify that (\ref{leftideal.condition}) and the remaining parts of (\ref{associativity.condition2}) hold. In other words it remains to show that
  \[
\left<\partial^{(k)} g | k \in \Z_+ \right>
\]
is closed under multiplication by $g$ and by $f$. This follows from Lemma \ref{zhu.closure} of the next section. The algebras are unital because of condition (\ref{non.higher.condition}).
\end{proof}

We denote by $\zhu_c(V)$ the algebra associated with (\ref{Zhu.solution}). It was shown by Huang {\cite[Proposition 6.3]{Huang.duality}} that $\zhu_{2\pi i}(V)$ is isomorphic to Zhu's algebra $A(V)$ as defined in \cite{Z96}. In fact the $\zhu_c(V)$ are isomorphic \cite{DK06} for all $c \neq 0$. The algebra associated with (\ref{C2.solution}) is Zhu's $C_2$-algebra, denoted $R_V$. The bracket $\{a,b\} = a(0)b$ is well defined on $R_V$ and makes it into a Poisson algebra.

%% Recall that a formal deformation of an associative $\C$-algebra $A$ is an assocative product $\circ_t$ on the $\C[[t]]$-algebra $A[[t]]$ that extends the product on $A$ in the sense that
%% \[
%% \left.{a \circ_t b}\right|_{t=0} = ab.
%% \]
%% A formal deformation of a commutative algebra $A$ gives it the structure of a Poisson algebra via
%% \[
%% \{a,b\} = \left.\frac{a \circ_t b - b \circ_t a}{t}\right|_{t=0},
%% \]
%% and the formal deformation is then said to be a deformation quantisation of the Poisson algebra $(A, \{\cdot,\cdot\})$.

Before proceeding we record a useful computation. Since $a_{(g)}\vac = a_{(-2)}\vac = -Ta$ we have $TV \subset V_{(g)}V$. The skew-symmetry formula (\ref{sksym}) implies in general that
\begin{align}\label{generalcomm}
  \begin{split}
  a_{(f)}b = \res_z f(z) e^{zT}b(-z)a \,dz
  = -b_{(f(-z))}a \mod{TV}.
  \end{split}
\end{align}
For the function (\ref{Zhu.solution}), which satisfies $f(-z) = -f(z) + c$, we obtain
\begin{align}\label{fcomm}
a_{(f)}b - b_{(f)}a = c \,\, a(0)b \mod{V_{(g)}V}.
\end{align}

Clearly the $c \rightarrow 0$ limit of (\ref{Zhu.solution}), i.e.,
\[
f(z) = z^{-1} + \frac{c}{2} + \frac{c^2}{12} z + \cdots
\]
recovers (\ref{C2.solution}).

\begin{rem}\label{higher.zhu.remark}
By relaxing condition (\ref{non.higher.condition}) one may obtain many additional solutions $(f, g)$ to the remaining relations, hence many additional associative, but no longer unital, quotients of $V$. If $g$ is taken to be $\partial f$ then replacement of the ideal $V_{(g)}V$ with $TV + V_{(g)}V$ permits one to dispense with condition (\ref{non.higher.condition}) while maintaining unitality of the quotient. This follows from the obvious inclusions
\[
(TV)_{(f)}V \subset V_{(\partial f)}V, \quad \text{and} \quad V_{(f)}(TV) \subset TV + (TV)_{(f)}V.
\]
In this way further associative quotients of $V$ are obtained, such as
\begin{align}\label{higher.zhu.half}
V/(TV + V_{(\partial f_{1/2})}V), \quad \text{equipped with the $f$-product, where} \quad f_{1/2}(z) = \frac{e^{2z}}{(e^z-1)^2}.
\end{align}
The higher Zhu algebras of Dong, Li and Mason \cite{DLM.VA.and.assoc} are obtained in a very similar way. For instance $\zhu_1(V)$ may be presented as $V / (TV + V_{(g)}V)$ equipped with the $f$-product, where
\[
g(z) = \frac{e^{2z}}{(e^z-1)^4}, \quad \text{and} \quad f(z) = \frac{e^{3z}-3e^{2z}}{(e^z-1)^3}.
\]
%In fact these algebras are obtained \cite{JVE.higher.Zhu} as quotients $V/J$ equipped with an $f$-product, where $J$ is the two-sided ideal generated under the $f$-product by $TV$.
\end{rem}

\section{Elliptic Functions}\label{section.elliptic}

In this section we recall some facts about elliptic function following the standard references {\cite{Apostol.Book}} and {\cite{Lang.Book.elliptic}}. We then complete the proof of Proposition \ref{zhu.is.assoc}.

%The fact that the solutions to (\ref{f.ODE}) satisfy the many other conditions (...) necessary to ensure that $J$ be an ideal, and associativity of $(V/J, *)$ seems surprising on the face of it. In fact it can be somewhat demystified in the context of elliptic functions.

Let $\HH$ denote the upper half complex plane. For some $\tau \in \HH$ let $\La = \La_\tau$ denote the lattice $\Z + \Z\tau \subset \C$. Then the Weierstrass elliptic function is defined to be
\begin{align*}
\wp(z,\tau) = \frac{1}{z^2} + \sum_{\la \in \La \backslash 0} \left(\frac{1}{(z-\la)^2} - \frac{1}{\la^2}\right).
  \end{align*}
It is $\La$-periodic and meromorphic on $\C$ with poles on $\La$. The ring of elliptic functions, i.e., of $\La$-periodic mermorphic functions on $\C$ with poles on $\La$, is spanned over $\C$ by $\ov\wp(z,\tau)$ and its derivatives in $z$ together with the constant function $1$. The Laurent series expansion of $\wp$ at $z=0$ takes the form
\begin{align}\label{wp.laurent}
\wp(z,\tau) = z^{-2} + \sum_{k=2}^\infty (2k-1) G_{2k}(\tau) z^{2k-2},
\end{align}
where the Eisenstein series $G_{2k}(\tau)$ is defined by
\begin{align}\label{eisenstein.fourier}
G_{2k}(\tau) = \sum_{n \in \Z \backslash 0} \frac{1}{n^{2k}} + \sum_{m \in \Z \backslash 0} \sum_{n \in \Z} \frac{1}{(m\tau+n)^{2k}}
\end{align}
for all $k \geq 1$. The well known Fourier expansion of $G_{2k}$ in powers of $q = e^{2\pi i \tau}$ is given by
\[
G_{2k}(\tau) = (-1)^{k+1}(2\pi)^{2k}\frac{B_{2k }}{(2k)!}E_{2k}, \quad \text{where} \quad E_{2k}=1- \frac{4k}{B_{2k }} \sum_{n=1}^\infty \frac{n^{2k-1} q^n}{1-q^n},
\]
and where the Bernoulli numbers $B_n$ are defined by $\sum_{n=0}^\infty B_n / n! = x / (e^x-1)$.

It will be more convenient for us to work with $\ov\wp(z,\tau)=\wp(z,\tau)+G_2(\tau)$, which can be expressed as
\begin{align}\label{wp.Lambert}
  \ov\wp(z,\tau) = (2\pi i)^2 \frac{e^{2\pi i z}}{(e^{2\pi i z}-1)^2} + (2\pi i)^2 \sum_{n=1}^\infty \frac{n q^n}{1-q^n} \left[e^{2\pi i n z} + e^{-2\pi i n z}\right],
\end{align}
on the domain of convergence $|\Imag(z)| < \Imag(\tau)$ of the right hand side.
% Verified with Mathematica.

The Weierstrass quasiperiodic function is defined to be
\begin{align*}
  \zeta(z,\tau) = \frac{1}{z} + \sum_{\la \in \La \backslash 0} \left( \frac{1}{z-\la} + \frac{z+\la}{\la^2} \right).
\end{align*}
Its Laurent series expansion is
\begin{align*}
  \zeta(z,\tau) = z^{-1} - \sum_{k=2}^\infty G_{2k}(\tau) z^{2k-1}.
\end{align*}
It will be more convenient for us to work with $\ov\zeta(z,\tau) = \zeta(z,\tau) + \pi i - G_2(\tau)z$, which satisfies
\begin{align}\label{zeta.Lambert}
\ov\zeta(z,\tau) = 2\pi i \frac{e^{2\pi i z}}{e^{2\pi i z}-1} - 2\pi i \sum_{n=1}^\infty \frac{q^n}{1-q^n} \left[ e^{2\pi i n z} - e^{-2\pi i n z} \right]
\end{align}
on the domain of convergence $|\Imag(z)| < \Imag(\tau)$ of the right hand side.
% Verified with Mathematica.

One has the relations
\begin{align}\label{zeta.qp}
\ov\zeta(z+1) = \ov\zeta(z) \quad \text{and} \quad \ov\zeta(z+\tau) = \ov\zeta(z) - 2\pi i.
\end{align}

\begin{lemma}\label{q.de}
The function
\begin{align}\label{formula.q.de}
(2\pi i)^2 q \partial_q \ov{\zeta} + \ov{\zeta} \cdot \ov{\wp}
\end{align}
is elliptic.
\end{lemma}

\begin{proof}
Let $2\pi i q \partial_q \ov\zeta(z, \tau) = \partial_\tau \ov\zeta(z, \tau)$ be denoted $\dot{\ov\zeta}(z, \tau)$. Applying $\partial_\tau$ to the second part of (\ref{zeta.qp}) yields
\[
\dot{\ov\zeta}(z+\tau,\tau) - \dot{\ov\zeta}(z,\tau) = -\partial_z\ov\zeta(z,\tau) = \ov\wp(z,\tau).
\]
On the other hand we have
\[
\ov\zeta(z+\tau,\tau)\ov\wp(z+\tau,\tau) - \ov\zeta(z,\tau)\ov\wp(z,\tau) = -2\pi i \ov\wp(z,\tau).
\]
From these two equations it follows that the sum (\ref{formula.q.de}) is elliptic.
\end{proof}
By differentiating (\ref{formula.q.de}) with respect to $z$ it follows that the functions
\begin{align}\label{formula.q.de.higher}
(2\pi i)^2 q \partial_q \partial_z^{k}\ov\wp - \ov\zeta \cdot \partial_z^{k+1}\ov\wp
\end{align}
are also elliptic. In \cite{JVEpisa} Ramanujan's differential equations on Eisenstein series were deduced from the $k=1$ case of this observation.

We now introduce the formal Laurent series
\begin{align*}
P(x,q) &= \frac{1+x}{x^2} + \sum_{n=1}^\infty \frac{nq^n}{1-q^n} \left[ (1+x)^{n} + (1+x)^{-n} \right] \\
\text{and} \quad Z(x,q) &= \frac{1+x}{x} - \sum_{n=1}^\infty \frac{q^n}{1-q^n} \left[ (1+x)^{n} - (1+x)^{-n} \right].
\end{align*}
Upon substitution of $x = e^{2\pi i z}-1$, these become convergent series on the domain $|\Imag(z)| < \Imag(\tau)$, and one has
\begin{align}\label{PZ.converge}
\ov\wp(z,\tau) = (2\pi i)^2 P(x,q) \quad \text{and} \quad \ov\zeta(z,\tau) = 2\pi i Z(x,q).
\end{align}

\begin{lemma}\label{zhu.closure}
The $\C[[q]]$-module $\left<\partial_z^{(k)}P | k \in \Z_+ \right>$ is closed under multiplication by $P$ and by $Z$ modulo $q\,\C[[q]]$. Let $f$ and $g$ be as in (\ref{Zhu.solution}). The vector space $\left<\partial_z^{(k)} g | k \in \Z_+ \right>$ is closed under multiplication by $f$ and by $g$.
\end{lemma}

\begin{proof}
The substitution $x = e^{cz}-1$ yields $f(z)=c(x^{-1}+1)$, $g(z)=-c^2(x^{-2}+x^{-1})$, and $\partial_z = c(x+1)\partial_x$. From this it is plain to see that $\partial_z^{(j)}g$, $f \partial_z^{(j)}g$ and $g \partial_z^{(j)}g$ lie in $x^{-1}\C[x^{-1}]$ for all $j \in \Z_+$.

Now let $\varepsilon : \C((x))[[q]] \rightarrow \C((x))$ denote the quotient by the ideal generated by $q$. We observe that $\varepsilon(Z) = f / c$ and $\varepsilon(P) = -g / c^2$.

Since $P$ is the Laurent expansion of an elliptic function it follows that, for each $j \in \Z_+$,
\[
P \partial_z^{(j)}P = \al + \sum_{k \in \Z_+} \beta_k \partial_z^{(k)}P,
\]
for some $\al, \beta_k \in \C[[q]]$. But since $\varepsilon(P) = -g / c^2$ and $g \partial_z^{(j)}g \in x^{-1}\C[x^{-1}]$ we see that $\varepsilon(\al) = 0$, hence $\al \in q\,\C[[q]]$. This proves at once the statements about closure with respect to $g$ and $P$.

Since $\varepsilon$ annihilates the image of $q \partial_q$ and the functions (\ref{formula.q.de.higher}) are elliptic, we may repeat the arguments above to deduce the statements about closure with respect to $f$ and $Z$ exactly as we did for $g$ and $P$.
\end{proof}

Aside from implying Proposition \ref{zhu.is.assoc}, the Lemma \ref{zhu.closure} is a source of classical identities among the constant terms of Eisenstein series. The following such identity is an example.
\begin{cor}
For each integer $n \geq 2$ one has
\begin{align*}
\left( 1 - \frac{n(n-1)}{2} \right) \frac{B_{n}}{n!} = \sum_{i=2}^{n-2} (i-1) \frac{B_{i}}{i!} \cdot \frac{B_{n-i}}{(n-i)!}.
\end{align*}

\end{cor}

It is easy to adapt the proof of Lemma \ref{zhu.closure} to show that the algebra defined in (\ref{higher.zhu.half}), henceforth referred to as $\zhu_{1/2}(V)$, is associative. We observe that $f_{1/2} = f + g$ and $g_{1/2} = f_{1/2}'$. The substitution $x = e^{z}-1$ yields $f_{1/2} = x^{-2} + 2x^{-1} + 1$ and $g_{1/2} = -2(x^{-3}+2x^{-2}+x^{-1})$. Now $\partial^{(j)}g_{1/2}$, $f_{1/2} \partial^{(j)}g_{1/2}$ and $g_{1/2} \partial^{(j)}g_{1/2}$ lie in $x^{-1}\C[x^{-1}]$ for all $j \in \Z_+$. The rest of the argument is identical. In fact this argument applies more generally. Put $F = f + p$ where $p$ is any linear combination of $g$ and its derivatives. Then the quotient $\zhu_F(V) = V/J_F$, where $J_F = V_{(F')}V + TV$, equipped with the $F$-product is a unital associative algebra.

\bibliographystyle{plain}

\end{document}